\documentclass[11pt]{amsart}
\usepackage{amsmath}
\usepackage{amsfonts}
\usepackage{amssymb}

\newcommand\eps{\varepsilon}
\newcommand\R{{\mathbb{R}}}
\newcommand\C{{\mathbb{C}}}

\theoremstyle{plain}
  \newtheorem{theorem}[subsection]{Theorem}

  \newtheorem{lemma}[subsection]{Lemma}

\theoremstyle{remark}
  \newtheorem{remark}[subsection]{Remark}

\theoremstyle{definition}

\begin{document}
\title[]{On characterization of the sharp Strichartz inequality for the Schr\"{o}dinger equation}
\author{Jin-Cheng Jiang}
\address{Department of Mathematics, National Tsing-Hua University, Hsinchu, Taiwan 30013, R.O.C}
\email{jcjiang@math.nthu.edu.tw}
\author{Shuanglin Shao}
\address{Department of Mathematics, University of Kansas, Lawrence, KS  66045}
\email{slshao@ku.edu}

\vspace{-1in}
\begin{abstract}
In this paper, we study the extremal problem for the Strichartz inequality for the Schr\"{o}dinger equation on the $\mathbb{R} \times \mathbb{R}^2$. We show that the solutions to the associated Euler-Lagrange equation are exponentially decaying in the Fourier space and thus can be extended to be complex analytic. Consequently we provide a new proof to the characterization of the extremal functions: the only extremals are Gaussian functions, which was investigated previously by Foschi \cite{Foschi:2007:maxi-strichartz-2d} and  Hundertmark-Zharnitsky \cite{Hundertmark-Zharnitsky:2006:maximizers-Strichartz-low-dimensions}.  
\end{abstract}
\date{\today}
\maketitle

\section{Introduction}

We begin with some notation. For a Schwarz function$f$ on $\R^d$, $d\ge 1$, define the Fourier transform,
$$\mathcal{F}(f)(\xi)= \widehat{f}(\xi) =\int_{\R^d} e^{-ix\cdot \xi} f(x)dx, \xi\in \R^d.$$
The inverse of the Fourier transform, 
$$\mathcal{F}^{-1}(f)(x)=f^\vee (x)=\frac {1}{(2\pi)^d}\int_{\R^d} e^{ix\cdot \xi} f(\xi)d\xi, x\in \R^d.$$  The linear Strichartz inequality for the Schr\"{o}dinger equation \cite{Keel-Tao:1998:endpoint-strichartz, Tao:2006-CBMS-book} asserts that 
\begin{equation}\label{eq-general-strichartz}
\|e^{it\Delta} f\|_{L^{2+\frac 4d}_{t,x}(\mathbb{R} \times \mathbb{R}^d)}\le C_d \|f\|_{L^2(\mathbb{R}^d)},
\end{equation}
where $e^{it\Delta}f (x) = \frac {1}{(2\pi)^d} \int_{\mathbb{R}^d} e^{ix\cdot \xi+it|\xi|^2} \hat{f}(\xi) d\xi. $ 
We specify $d=2$ and consider 
\begin{equation}\label{eq-strichartz}
\|e^{it\Delta} f\|_{L^4_{t,x}(\mathbb{R} \times \mathbb{R}^2)}\le {\bf R}\|f\|_{L^2(\mathbb{R}^2)}. 
\end{equation} 
where
\begin{equation}\label{eq-17}
{\bf R}:=\sup\left\{\frac {\|e^{it\Delta} f \|_{L^4_{t,x}(\R\times \R^2)}}{ \|f\|_{L^2(\R^2)}}: f\in L^2, \, f\neq 0 \right\}.
\end{equation}

We define an extremal function or extremal to \eqref{eq-strichartz} is a nonzero function $f\in L^2$ such that the inequality is optimized in the sense that 
\begin{equation} \|e^{it\Delta} f\|_{L^4_{t,x}(\mathbb{R} \times \mathbb{R}^2)}={\bf R}\|f\|_{L^2(\mathbb{R}^2)}. 
\end{equation}
The extremal problem of \eqref{eq-strichartz} concerns (i) Whether there exists an extremal function? (ii) How to characterize the extremal functions? What are the explicit forms of extremal functions?  Are they unique up to the symmetry of the inequality? 

From Foschi \cite{Foschi:2007:maxi-strichartz-2d} and Hundertmark-Zharnitsky \cite{Hundertmark-Zharnitsky:2006:maximizers-Strichartz-low-dimensions}, it is known that, the Gaussian functions are the only extremal functions to the linear Strichartz inequality \eqref{eq-strichartz} for the dimensions $d=1,2$. Here Gaussian functions, $\mathbb{R}^d\to \mathbb{C}$, $d=1,2$, are of the form 
$$e^{A |x|^2+ B\cdot x + C},$$
with $A, C \in \mathbb{C}, B\in \mathbb{C}^d$ and the real part of $A$, $\Re (A)<0$. The existence of extremisers was established previously by Kunze \cite{Kunze:2003:maxi-strichartz-1d} for the Strichartz inequality \eqref{eq-general-strichartz} when $d=1$. When $d\ge 3$, existence of extremisers is proved by the second author \cite{Shao:2009:maximizers-schrodinger} . 

In this note, we are interested in the problem of how to characterize extremals for \eqref{eq-strichartz} via the study of the associated Euler-Lagrange equation. We show that the solutions of this generalized Euler-Lagrange equation enjoy a fast decay in the Fourier space and thus can be extended to be complex analytic, see Theorem \ref{thm-complex-analyticity}. Then as an easy consequence, we give an alternative proof that all extremal functions to \eqref{eq-strichartz} are Gaussians based on solving a functional equation of extremizers derived in Foschi \cite{Foschi:2007:maxi-strichartz-2d}, see \eqref{eq-multiplicative} and Theorem \ref{thm-quadratic}. Indeed, in the proof given below we use the information that $f$ is twice continuously differentiable, i.e., $f\in C^2$, which can be lowered to continuity by a more refined argument. The functional inequality \eqref{eq-multiplicative} is a key ingredient in Foschi's proof in \cite{Foschi:2007:maxi-strichartz-2d}. To prove $f$ in \eqref{eq-multiplicative} to be a Gaussian function, local integrability of $f$ is assumed in \cite{Foschi:2007:maxi-strichartz-2d}, which is further reduced to measurable functions in Charalambides \cite{Charalambides:2012:Cauchy-Pexider}.

Let $f$ be an extremal function to \eqref{eq-strichartz} with the constant ${\bf R}$. Then $f$ satisfies the following generalized Euler-Lagrange equation,
\begin{equation}\label{eq-strichartzequ}
\omega \langle g, f \rangle = \mathcal{Q}(g,f,f,f), \text{for all }g \in L^2,
\end{equation}
where $\omega= \mathcal{Q}(f,f,f,f)/\|f\|_{L^2}^2 >0$ and $\mathcal{Q}(f_1,f_2,f_3,f_4)$ is the integral
\begin{equation}\label{eq-21} 
\begin{split}
& \int_{(\R^2)^4} \overline{\widehat{f_1}}(\xi_1) \overline{\widehat{f_2}}(\xi_2)  \widehat{f_3}(\xi_3)  \widehat{f_4}(\xi_4) \delta (\xi_1 + \xi_2 -\xi_3-\xi_4) \times \\
&\qquad\qquad \qquad \times \delta(|\xi_1|^2 + |\xi_2|^2 -|\xi_3|^2- |\xi_4|^2) d\xi_1 d\xi_2 d\xi_3 d\xi_4,
\end{split}
\end{equation}
for $f_i \in L^2(\R^2)$, $1\le i\le 4$, $\delta (\xi) = (2\pi)^{-d} \int_{\mathbb{R}^d} e^{i \xi \cdot x}dx$ in the distribution sense, $d=1,2$. The proof of \eqref{eq-strichartzequ} is standard; see e.g.  \cite[p. 489]{Evans:2010:PDE-book} or \cite[Section 2]{Hundertmark-Lee:2012:nonlocal-variational-problems-nonlinear-optics} for similar derivations of Euler-Lagrange equations. 

\begin{theorem}\label{thm-complex-analyticity}
If $f$ solves the generalized Euler-Lagrange equation \eqref{eq-strichartzequ} for some $\omega>0$, then there exists $\mu>0$ such that 
$$ e^{\mu |\xi|^2} \widehat{f} \in L^2(\mathbb{R}^2). $$
Furthermore $f$ can be extended to be complex analytic on $\C^2$.  
\end{theorem}

To prove this theorem, we follow the argument in \cite{Hundertmark-Shao:Analyticity-of-extremals-Airy-Strichartz}. Similar reasoning has appeared previously in \cite{Erdogan-Hundertmark-Lee:2011:exponential-decay-of-dispersion-management-soliton,Hundertmark-Lee:2009:decay-smoothness-for-solutions-dispersion-managed-NLS}. It relies on a multilinear weighted Strichartz estimate and a continuity argument. See Lemma \ref{le-multilinear} and Lemma \ref{le-bootstrap}. 

Next we prove that the extremals to \eqref{eq-strichartz} are Gaussian functions. We start with the study of the functional equation derived in \cite{Foschi:2007:maxi-strichartz-2d}. In \cite{Foschi:2007:maxi-strichartz-2d}, the functional equation reads
\begin{equation}\label{eq-multiplicative}
f(x)f(y) = f(w)f(z),
\end{equation}
for any $x,y,w,z \in \mathbb{R}^2$ such that 
\begin{equation}\label{eq-parallelgram}
x+y = w+ z, \quad |x|^2 + |y|^2 = |w|^2 + |z|^2,
\end{equation}
Note that $x, y, w, z$ in $\mathbb{R}^2$ satisfy the relation \eqref{eq-parallelgram} if and only if these four points form a rectangle in $\mathbb{R}^2$ with vertices $x,y,w$ and $z$.  Indeed, by \eqref{eq-parallelgram}, these four points $x, y ,w$ and $z$ form a parallelogram on $\mathbb{R}^2$ and $x\cdot y = w \cdot z$ . Secondly $w-x$ is perpendicular to $z-x$ since $(w-x) \cdot (z-x) = w\cdot z - w \cdot x - x \cdot z + |x|^2 = w\cdot z -(x+y) \cdot x + |x|^2 = w \cdot z - y \cdot x =0.$ This proves that $x,y,w$ and $z$ form a rectangle on $\mathbb{R}^2$. In \cite{Foschi:2007:maxi-strichartz-2d}, it is proven that $f\in L^2$ satisfies \eqref{eq-multiplicative} if and only if $f$ is an extremal function to \eqref{eq-strichartz}. Basically, this comes from two aspects. One is that in the Foschi's proof of the sharp Strichartz inequality only the Cauchy-Schwarz inequality is used at one place besides equality. So the equality in the Strichartz inequality \eqref{eq-strichartz}, or equivalently the equality in Cauchy-Schwarz,  yields the same functional equation as \eqref{eq-multiplicative} where $f$ is replaced by $\hat{f}$. The other one is that the Strichartz norm for the Schr\"odinger equation enjoys an identity that \begin{equation}\label{eq-symmetry}
\|e^{it\Delta}f \|_{L^4(\mathbb{R}\times \mathbb{R}^2)} =C  \|e^{it\Delta}f^\vee \|_{L^4(\mathbb{R}\times \mathbb{R}^2)}
\end{equation}for some $C>0$.

In \cite{Foschi:2007:maxi-strichartz-2d}, Foschi is able to show that all the solutions to \eqref{eq-multiplicative} are Gaussians under the assumption that $f$ is a locally integral function. This can be viewed as an investigation of the Cauchy functional equation \eqref{eq-multiplicative} for functions supported on the paraboloids. To characterize the extremals for the Tomas-Stein inequality for the sphere in $\mathbb{R}^3$, in \cite{Christ-Shao:extremal-for-sphere-restriction-II-characterizations},  Christ and the second author study the same functional equation \eqref{eq-multiplicative} for functions supported on the sphere and prove that they are exponentially affine functions.  In \cite{Charalambides:2012:Cauchy-Pexider}, Charalambides generalizes the analysis in \cite{Christ-Shao:extremal-for-sphere-restriction-II-characterizations} to some general hyper-surfaces in $\mathbb{R}^n$ that include the sphere, paraboloids and cones as special examples and proves that the solutions are exponentially affine functions. In \cite{Charalambides:2012:Cauchy-Pexider, Christ-Shao:extremal-for-sphere-restriction-II-characterizations}, the functions are assumed to be measurable functions.  

By the analyticity established in Theorem \ref{thm-complex-analyticity}, Equations \eqref{eq-multiplicative} and \eqref{eq-parallelgram} have the following easy consequence, which recovers the result in \cite{Foschi:2007:maxi-strichartz-2d, Hundertmark-Zharnitsky:2006:maximizers-Strichartz-low-dimensions}. 

\begin{theorem}\label{thm-quadratic}
Suppose that $f$ is an extremal function to \eqref{eq-strichartz}. Then 
\begin{equation}\label{eq-quadratic}
f(x) = e^{A |x|^2+ B \cdot x + C},
\end{equation} 
where $A, C \in \mathbb{C}, B \in \mathbb{C}^2$ and $\Re(A)<0$. 
\end{theorem}

Let $f$ be an extremal function to \eqref{eq-strichartz}. Then by Theorem \ref{thm-complex-analyticity}, $f$ is continuous. This, together with \eqref{eq-multiplicative} and $\eqref{eq-parallelgram}$, implies that any nontrivial $f$ is nowhere vanishing on $\mathbb{R}^2$, see e.g. \cite[Lemma 7.13]{Foschi:2007:maxi-strichartz-2d}. For any $a\in \R^2$,  there is a disk $D(a,r) \subset \C^2$, $r>0$, such that $f$ is $C^2$ by Theorem \ref{thm-complex-analyticity} and $f$ is nowhere vanishing. Then $\log f$ is $C^2$ on $D(a,r)$, see e.g. \cite[Lemma 6.1.9]{Krantz:1992:complex-analysis-several-variables}.  Similar claims can be made for $\log f^2$.  Then up to a multiple of $2\pi$, 
$$ \log f^2(a) = \log f(a)+ \log f(a). $$
After restriction to $\R^2$, $f$ satisfies the equation \eqref{eq-multiplicative} for $x,y,w,z$ satisfying \eqref{eq-parallelgram}.  So by taking $r$ sufficiently small, 
$$\log f(x) + \log f(y) = \log f(w)+\log f (z)$$
for $x,y,w, z$ in $B(a,r)\subset \R^2$ and related as in \eqref{eq-parallelgram}. Since $\log f$ is twice differentiable, it is not hard to see that $\log f$ is a quadratic polynomial on $B(a,r)$. So $\log f $ is a quadratic polynomial on $\R^2$. Indeed, Let $a=0$ and $\phi(x_1)= \log f(x_1,0), \psi(0,x_2) = \log f(0,x_2)$. Then since the four points $(x_1,x_2), (x_2,-x_1) $ and $(x_1+x_2,x_2-x_1), (0,0)$ satisfy \eqref{eq-parallelgram}, we see that 
$$[\phi(x_1)+\psi(x_2)]+[\phi(x_2)+\psi(-x_1)] = [\phi(x_1+x_2)+\psi(x_2-x_1)] +\log f (0,0).$$
By taking differentiation firstly in $x_1$ and then in $x_2$, we see that $\phi '' = \psi ''$ is a constant. Thus $f$ is a quadratic polynomial. It is easy to see that this argument generalizes to any $a\in \R^2$.

{\bf Acknowledgement.} The research of the first author is supported by National Science Council Grant NSC 102-2115-M-007-0101-MY 2. The second author is supported in part by the NSF grant DMS-1160981. The authors would like to thank the anonymous referee for helpful comments and suggestions, which have been incorporated into this paper. 

 \section{Complex Analyticity}\label{sec:complex-analyticity}
In this section, we show that the solutions to the generalized Euler-Lagrange equation \eqref{eq-strichartzequ} can be extended to be complex analytic. 
 
We define
\begin{align*}
 \eta &: = (\eta_1, \eta_2, \eta_3, \eta_4 ) \in (\R^2)^4, \\
 a(\eta) & := \eta_1 + \eta_2 -\eta_3 -\eta_4, \\
 b(\eta ) & := |\eta_1|^2+ |\eta_2|^2 - |\eta_3|^2- |\eta_4|^2.  
\end{align*}
Let  $\eps\ge 0$ and $ \mu\ge 0$. For $\xi\in \R^2$, define
\begin{equation}
F(\xi) := F_{\mu, \eps} (\xi) = \frac {\mu |\xi|^2}{1+\eps |\xi|^2}. 
\end{equation} 
Define the weighted multilinear integral, for $h_i \in L^2(\mathbb{R}^2)$, $1\le i\le 4$, 
\begin{equation}\label{eq-22}
M_F(h_1, h_2, h_3, h_4) :=  \int_{(\R^2)^4} e^{F(\eta_1) - \sum_{j=2}^4 F(\eta_j)} \Pi_{j=1}^4 |h(\eta_j)| \delta\bigl( a(\eta)\bigr) \delta\bigl(b(\eta)\bigr) d\eta.
\end{equation}
The multilinear estimate we need shows the weak interaction of Schr\"{o}dinger waves between the high and low frequency.  More precisely,
\begin{lemma}\label{le-multilinear}
Let $h_i\in L^2(\R^2)$, $1\le i\le 4$, and $s>1$ be a large number. The Fourier transforms of $h_1, h_2$ are supported in $\{\xi: \, |\xi| \le s\}$,  and $\{\xi: \, |\xi|\ge Ns \}$ with $N >1$ being a large number, respectively. Then 
\begin{equation}\label{eq-23}
M_F(h_1, h_2, h_3, h_4)  \le C N^{-1/2} \Pi_{j=1}^4 \|h_j\|_{L^2}. 
\end{equation} 
\end{lemma}
\begin{proof}
The proof of this lemma needs the following two inequalities, 
\begin{equation}\label{eq-24}
M_F(h_1, h_2, h_3, h_4)  \le \int_{(\R^2)^4} \Pi_{j=1}^4 |h_j(\eta_j)|  \delta\bigl( a(\eta)\bigr) \delta\bigl(b(\eta)\bigr) d\eta;
\end{equation}
and 
\begin{equation}\label{eq-25}
\|e^{it\Delta} h_1 e^{it\Delta} h_2 \|_{L^2_{t,x}} \le C N^{-1/2} \|h_1\|_{L^2}\|h_2\|_{L^2}. 
\end{equation}
Together with the Cauchy-Schwarz inequality and the $L^2\to L^4 $ Strichartz inequality, the inequality \eqref{eq-23} follows from \eqref{eq-24} and \eqref{eq-25}.  Note that \eqref{eq-25} is established in \cite{Bourgain:1998:refined-Strichartz-NLS}. Thus it remains to establishing \eqref{eq-24}, where we follow \cite{Erdogan-Hundertmark-Lee:2011:exponential-decay-of-dispersion-management-soliton,Hundertmark-Shao:Analyticity-of-extremals-Airy-Strichartz}. 

On the support of $\eta$ determined by $\delta (a(\eta))$ and $\delta(b(\eta))$, we have 
\begin{align*}
\eta_1 + \eta_2 &= \eta_3+ \eta_4, \\
|\eta_1|^2+|\eta_2|^2 &= |\eta_3|^2+ |\eta_4|^2.
\end{align*}
Thus
$$|\eta_1|^2 \le |\eta_2|^2+ |\eta_3|^2+ |\eta_4|^2.$$
Since the function $x \to \frac {x}{1+ \epsilon x}$ is increasing on the interval $[0, \infty)$, we have 
$$\frac {|\eta_1|^2}{1+ \eps |\eta_1|^2} \le \frac {\sum_{j=2}^4 |\eta_j|^2}{1+\sum_{j=2}^4 \eps |\eta_j|^2} = \sum_{j=2}^4 \frac {|\eta_j|^2}{1+\sum_{j=2}^4 \eps |\eta_j|^2 } \le \sum_{j=2}^4\frac {|\eta_j|^2}{1+\eps |\eta_j|^2}. $$
This implies that $F(\eta_1) \le \sum_{j=2}^4 F(\eta_j)$ since $\mu \ge 0$.  Hence $$e^{F(\eta_1) - \sum_{j=2}^4 F(\eta_j)} \le 1. $$
Therefore \eqref{eq-24} follows by taking the absolute value in the integral. 
\end{proof}

If $f\in L^2$ satisfies the generalized Euler-Lagrange equation \eqref{eq-strichartzequ}, the following bootstrap lemma shows that $f$ gains certain regularity, namely, there is a constant $\mu>0$ depending on the function $f$, $e^{\mu|\xi|^2} \widehat{f}\in L^2. $ This is enough to conclude that $f$ can be extended to be complex analytic. 
\begin{lemma}\label{le-bootstrap}
If $f$ solves the generalized Euler-Lagrange equation \eqref{eq-strichartzequ} for some $\omega>0$ and $\|f\|_{L^2}=1$, then for $\widehat{f}_{>}:= \widehat{f} 1_{|\xi|\ge s^2}$ for $s>0$, there is a large constant $s\gg 1$ such that for $\mu= s^{-4}$, 
\begin{equation}\label{eq-27}
\omega \|e^{F(\cdot)} \widehat{f}_{>}\|_{L^2}\le o_1(1)  \|e^{F(\cdot)} \widehat{f}_{>}\|_{L^2} + C \|e^{F(\cdot)} \widehat{f}_{>}\|_{L^2}^2+C \|e^{F(\cdot)} \widehat{f}_{>}\|_{L^2}^3 + o_2(1),
 \end{equation}
 where $\lim_{s\to \infty}o _i(1)=0$ uniformly for all $\eps>0$, $i=1,2$, the constant $C>0$ is independent of $\eps$ and $s$. 
\end{lemma}

\begin{proof}
Define $h(\xi) = e^{F(\xi)} \widehat{f}(\xi) $ and $h_>(\xi) = e^{F(\xi)} \widehat{f}_{>}$, where $\widehat{f}_{>} = \widehat{f}1_{|\xi| \ge s^2}$.  Let $P$ denote the symbol of differentiation of $ -i\partial_x$; under the Fourier transform, $\widehat{P} = |\xi|$. Correspondingly, we write $F(P)$ with the Fourier symbol $\frac {\mu |\xi|^2}{1+\eps |\xi|^2}$. 

We expand 
\begin{equation*}
 \|e^{F(\cdot)} \widehat{f}_{>}\|_{L^2}^2 =  \langle e^{F(\cdot)} \widehat{f}_{>}, e^{F(\cdot)} \widehat{f}_{>} \rangle = 
  \langle e^{2F(\cdot)} \widehat{f}_{>}, \widehat{f} \rangle = \langle e^{2F(P)}f_>, f \rangle.  
\end{equation*}
Thus in the generalized Euler-Lagrange \eqref{eq-strichartzequ}, setting $g = e^{2F(P)}f_{>}$, we see that 
\begin{equation}\label{eq-28}
\omega  \|e^{F(P)} f_{>}\|_{L^2}^2 = Q (e^{2F(P)} f_>, f,f,f).
\end{equation}
Since $\widehat{f}= e^{-F(\xi)} h$ and $e^{2F(\xi)} \widehat{f}_> = e^{F(\xi)} h_>$, 
\begin{align*}
&Q (e^{2F(P)} f_>, f,f,f) = \int_{(\R^2)^4} e^{2F(\xi_1)} \overline{\widehat{f}_>}(\xi_1) 
\overline{\widehat{f}_>}(\xi_2) \widehat{f}(\xi_3) \widehat{f}_4(\xi_4) \delta (a(\xi)) \delta (b(\xi)) d\xi \\
&=\int_{(\R^2)^4}\overline{e^{F(\xi_1)} h_>(\xi_1)}\, \overline{e^{-F(\xi_2)} h(\xi_2)} e^{-F(\xi_3)} h(\xi_3)  e^{-F(\xi_4)} h(\xi_4) \delta (a(\xi)) \delta (b(\xi)) d\xi \\
&= \int_{(\R^2)^4} e^{F(\xi_1) - \sum_{j=2}^4 F(\xi_j)} h_>(\xi_1) h(\xi_2) h(\xi_3) h(\xi_4) \delta (a(\xi)) \delta (b(\xi))  d\xi,
\end{align*}
where  $a(\xi)= \xi_1+\xi_2-\xi_3-\xi_4$ and $b(\xi) = |\xi_1|^2+|\xi_2|^2-|\xi_3|^2-|\xi_4|^2$ for $\xi= (\xi_1,\xi_2,\xi_3,\xi_4)\in \bigl(\R^2\bigr)^4$. Thus
\begin{equation}\label{eq-29}
\omega   \|e^{F(P)} f_{>}\|_{L^2}^2 \le M_F(h_>, h, h, h).
\end{equation}
Define 
\begin{equation*}
h_{\sim} = h1_{s\le |\xi| \le s^2},  h_{<<} = h1_{|\xi| < s} \text{ and } h_<= h_{<<}+ h_{\sim}. 
\end{equation*}
We split the integral $M_F (h_>, h, h, h)$ into the following pieces, 
\begin{equation*}
M_F (h_>, h_<, h_<, h_<) + \sum_{j_2, j_3, j_4} M_F(h_>, h_{j_2}, h_{j_3}, h_{j_4}) =:A+B,
\end{equation*}where $h_{j_k}$ is either $h_>$ or $h_<$, but at least one is $h_>$. 
We further split $A$ into two terms,
\begin{equation*}
M_F (h_>, h_{<<}, h_<, h_<)+ M_F (h_>, h_\sim, h_<, h_<);
\end{equation*}
we estimate this term by using Lemma \ref{le-multilinear}, 
\begin{equation*} 
\begin{split}
A &\lesssim s^{-1/2} \|h_>\|_{L^2} \|h_{<<}\|_{L^2} \|h_<\|^2_{L^2}+ \|h_>\|_{L^2} \|h_\sim\|_{L^2} \|h_<\|^2_{L^2} \\
&\lesssim \|h_>\|_{L^2} \left( s^{-1/2} \|h_{<<}\|_{L^2} + \|h_\sim \|_{L^2} \right) \|h_<\|^2_{L^2}.
\end{split}
\end{equation*}
Since $\|f\|_{L^2}=1$, 
\begin{align*}
\|h_<\|_{L^2} &\le  e^{\mu s^4} \|f\|_{L^2} = e^{\mu s^4}, \\
\|h_{<<}\|_{L^2} &\le e^{\mu s^2}, \\
\|h_\sim \|_{L^2} &\le e^{\mu s^4} \|f_\sim\|_{L^2},
\end{align*}where $f_\sim $ is defined, $\widehat{f}_\sim = \widehat{f} 1_{s\le | \xi |\le s^2}$. Thus we have 
\begin{equation}\label{eq-30}
A \lesssim e^{3\mu s^4}\|h_>\|_{L^2} \left( s^{-1/2} e^{\mu s^2 -\mu s^4} + \|f_\sim \|_{L^2}\right).
\end{equation}
Similarly we estimate the term $B$. We split $B$ into two terms $B_1+ B_2$, where $B_1=\sum_{j_2, j_3, j_4} M_F (h_>, h_{j_2}, h_{j_3}, h_{j_4})$ containing exactly one $h_>$ in $\{h_{j_2}, h_{j_3}, h_{j_4}\}$, while $B_2=\sum_{j_2, j_3, j_4} M_F (h_>, h_{j_2}, h_{j_3}, h_{j_4}) $ containing two or more $h_>$. 

To estimate $B_1$, 
\begin{equation}\label{eq-31}
\begin{split}
B_1 & \lesssim e^{\mu s^4}\|h_>\|^2_{L^2} \|h_<\|_{L^2} \left( s^{-1/2} e^{\mu s^2 -\mu s^4}+ \|f_\sim \|_{L^2}\right) \\
&\lesssim e^{2\mu s^4}\|h_>\|^2_{L^2} \left( s^{-1/2} e^{\mu s^2 -\mu s^4}+ \|f_\sim \|_{L^2}\right).
\end{split}
\end{equation}
To estimate $B_2$, 
\begin{equation}\label{eq-32}
B_2 \lesssim \|h_>\|^3_{L^2} \|h_<\|_{L^2}+ \|h_>\|^4_{L^2}\lesssim e^{\mu s^4} \|h_>\|^3_{L^2}+ \|h_>\|^4_{L^2}.  
\end{equation}
Thus from \eqref{eq-30}, \eqref{eq-31} and \eqref{eq-32}, we obtain
\begin{align*}
 \|e^{F(\cdot)} \widehat{f}_{>}\|_{L^2}^2 
&\lesssim e^{3\mu s^4}\|h_>\|_{L^2} \left( s^{-1/2} e^{\mu s^2 -\mu s^4}+ \|f_\sim \|_{L^2}\right)  \\
&\qquad +  e^{2\mu s^4}\|h_>\|^2_{L^2} \left( s^{-1/2} e^{\mu s^2 -\mu s^4}+ \|f_\sim \|_{L^2}\right) \\
&\qquad +e^{\mu s^4} \|h_>\|^3_{L^2}+\|h_>\|^4_{L^2}. 
\end{align*}
Since $\lim_{s\to \infty}\|f_\sim \|_{L^2} =0$, we take $s$ sufficiently large and set $\mu = s^{-4}$, 
\begin{equation}\label{eq-33}
\omega \|e^{F(\cdot)} \widehat{f}_{>}\|_{L^2} \le o_1(1)  \|e^{F(\cdot)} \widehat{f}_{>}\|_{L^2} + C \|e^{F(\cdot)} \widehat{f}_{>}\|_{L^2}^2+C \|e^{F(\cdot)} \widehat{f}_{>}\|_{L^2}^3+ o_2(1), 
\end{equation}
which completes the proof of Lemma \ref{le-bootstrap}. 
\end{proof}
\begin{remark}
Clearly the choice of $\mu$ in the preceding lemma depends on the function $f$ itself.  
\end{remark}
Now we conclude that $f$ in Lemma \ref{le-bootstrap} gains certain regularity. 

\begin{proof}[Proof of Theorem \ref{thm-complex-analyticity}]
Let $f\in L^2$ and $f\neq 0$. We normalize $f$ such that $\|f\|_{L^2} =1$. In Lemma \ref{le-bootstrap}, we choose $s$ sufficiently large such that $o_1(1) \le  {\omega}/{2}$ and $o_2(1) \le M/2$, where $M = \sup\{ G(x): \, x\in [0,\infty)\}$, and 
\begin{equation}
G(x) := \frac \omega 2 x-C x^2-Cx^3, \, x\in [0,\infty),
\end{equation}
and $C$ is the same constant as in \eqref{eq-27}. It is easy to see that $0\le M<\infty$.  Then $G (x) \le M$ for all $x\in [0,\infty)$ by Lemma \ref{le-bootstrap}.  Also the function $G$ is continuous on $[0,\infty)$. On the other hand, $G'' (x) <0$ for all $x\in (0,\infty)$; thus $G$ is concave. The line $G= \frac M2$ intersects at two points of the positive $x$ axis, $x=x_0$ and $x= x_1>0$. 

We define $H:(0,\infty) \to [0,\infty)$ via $$H(\epsilon) = \left(\int_{|\xi|\ge s^2} \left| e^{F_{s^{-4},\epsilon}(\xi)} \widehat{f}\right|^2 d\xi \right)^{1/2}.$$
The function $H$ is continuous on $(0,\infty)$ by the dominated convergence theorem and $H(0,\infty)$ is connected.  Hence $ G^{-1} ([0, \frac M2])$ is either contained in $[0,x_0]$ or contained in $[x_1,\infty)$; only one alternative holds. For $\epsilon =1$ and $s$ sufficiently large, $H(1) \ge x_1$ is impossible. Hence the first alternative holds.

Therefore $G^{-1}([0, \frac M2]) \subset [0,x_0]$, which yields that  
\begin{equation}\label{eq-35}
 \|e^{F(\cdot)} \hat{f}_{>}\|_{L^2}\le C_0. \text{ i.e., } \left\| e^{\frac {s^{-4}|\xi|^2}{1+\epsilon |\xi|^2}}\hat{f}_>\right\|_{L^2}\le C_0,
\end{equation} uniformly in all $\eps>0$. By the monotone convergence theorem, 
$$\|e^{s^{-4} |\xi|^2}\widehat{f}_{>} \|_{L^2} \le C_0<\infty.$$
It is clearly that $e^{s^{-4} |\xi|^2} \widehat{f}1_{|\xi|\le s^2} \in L^2.$ Therefore, 
$$ e^{s^{-4} |\xi|^2} \widehat{f} \in L^2. $$
Let $\mu = s^{-4}$. This proves the first half of Theorem \ref{thm-complex-analyticity}. 

To prove that $f$ can be extended to be complex analytic on $\mathbb{C}^2$, we observe that, by the Cauchy-Schwarz inequality, for any $\lambda \in \mathbb{R}$, 
\begin{equation}\label{eq-84}
e^{\lambda |\xi|}\widehat{f}(\xi)=e^{\lambda |\xi|-\mu|\xi|^2} e^{\mu |\xi|^2} \widehat{f}(\xi) \in L^2(\mathbb{R}^2). 
\end{equation}
So it is not hard to see that $f$ can be extended to be complex analytic on $\mathbb{C}^2$, see e.g. \cite[Theorem IX.13]{Reed-Simon:1975:book2}. Alternatively, analyticity can be obtained in the following way. Similarly as in \eqref{eq-84} for $k\in \mathbb{N}\cup \{0\}$,
$ |\xi|^ke^{\lambda |\xi|}\widehat{f} \in L^1(\mathbb{R}^2)$. For $z\in \mathbb{C}^2$, we choose $\lambda >|z|$, 
$$f(z) = (2\pi)^{-2} \int_{\mathbb{R}^2} e^{iz\cdot \xi -\lambda |\xi|} e^{\lambda |\xi|} \widehat{f} (\xi)d\xi. $$
Then by taking differentiation under the integral sign, complex analyticity follows.   
\end{proof}

\end{document}